\documentclass{amsart}

\usepackage[T2A,T1]{fontenc}
\usepackage[lutf8]{luainputenc}
\usepackage[russian,USenglish]{babel}
\usepackage{amsmath,amssymb,amsthm,upgreek}
\usepackage{colortbl,color,xcolor} 
\usepackage{graphicx,tikz,pgfplots} 
\usepackage{enumitem}
\pgfplotsset{compat=1.6}
\usepackage{subfigure,float,longtable}
\usepackage[bookmarks=true,bookmarksnumbered,plainpages=false,linktocpage,colorlinks=true,citecolor=green!80!black,linkcolor=red!70!black,filecolor=magenta,urlcolor=magenta,hidelinks,breaklinks,pdfauthor={Thomas Jahn, Christian Richter},pdftitle={Coproximinality of linear subspaces in generalized Minkowski spaces}]{hyperref}
\usepackage[nameinlink,capitalize,noabbrev]{cleveref}
\usepackage[babel]{csquotes}

\usetikzlibrary{decorations.fractals,calc,intersections,through,backgrounds,arrows,patterns,shapes.geometric,shapes.misc,spy}

\newcommand{\rus}[1]{\selectlanguage{russian}{\fontfamily{erewhon-TLF}\fontsize{9pt}{11pt}\selectfont #1}\selectlanguage{USenglish}}

\newcommand{\RR}{\mathbb{R}}                                     
\newcommand{\abs}[1]{\left\lvert#1\right\rvert}                  
\newcommand{\mnorm}[1]{\left\lVert#1\right\rVert}                
\newcommand{\setn}[1]{\left\{#1\right\}}                         
\newcommand{\setcond}[2]{\left\{#1 \:\middle\vert\: #2\right\}}  
\newcommand{\defeq}{\mathrel{\mathop:}=}                         
\newcommand{\lr}[1]{\!\left(#1\right)}                           
\newcommand{\mo}[3][]{U_{#1}(#2,#3)}                             
\newcommand{\mc}[3][]{B_{#1}\left(#2,#3\right)}                             
\newcommand{\ms}[3][]{S_{#1}(#2,#3)}                             
\newcommand{\strline}[2]{\left\langle#1,#2\right\rangle}         
\newcommand{\clray}[2]{\left[#1,#2\right\rangle}                 
\newcommand{\skpr}[2]{\left\langle#1 \,\middle\vert\, #2\right\rangle}               
\newcommand{\eqdef}{=\mathrel{\mathop:}}                                             
\DeclareMathOperator{\bisec}{bsc}
\DeclareMathOperator{\CONE}{cone}
\newcommand{\cone}[1]{\CONE\left(#1\right)}

\DeclareMathOperator{\LIN}{lin}
\newcommand{\lin}[1]{\LIN\left(#1\right)}

\theoremstyle{plain} 
\newtheorem{Satz}{Theorem}[section]\newtheorem{Kor}[Satz]{Corollary}
\newtheorem{Lem}[Satz]{Lemma}
\newtheorem{Prop}[Satz]{Proposition}

\theoremstyle{definition} 
\newtheorem{Def}[Satz]{Definition}
\newtheorem{Bsp}[Satz]{Example}

\crefname{Satz}{Theorem}{Theorems}
\crefname{Prop}{Proposition}{Propositions}
\crefname{Lem}{Lemma}{Lemmas}
\crefname{Kor}{Corollary}{Corollaries}
\crefname{Bem}{Remark}{Remarks}
\crefname{Bsp}{Example}{Examples}
\crefname{Def}{Definition}{Definitions}

\numberwithin{equation}{section}

\makeatletter
\renewcommand*{\eqref}[1]{%
  \hyperref[{#1}]{\textup{\tagform@{\ref*{#1}}}}%
}
\makeatother

\makeatletter
\@ifundefined{textcommabelow}{%
  \DeclareTextCommandDefault\textcommabelow[1]
    {\hmode@bgroup\ooalign{\null#1\crcr\hidewidth\raise-.31ex
     \hbox{\check@mathfonts\fontsize\ssf@size\z@
     \math@fontsfalse\selectfont,}\hidewidth}\egroup}%
}{}
\makeatother

\begin{document}
\allowdisplaybreaks
\parindent 0pt

\title{Coproximinality of linear subspaces in generalized Minkowski spaces}


\author{Thomas Jahn}
\address[T. Jahn]{Faculty of Mathematics, Technische Universität Chemnitz, 09107 Chemnitz, Germany}
\email{thomas.jahn@mathematik.tu-chemnitz.de}
\thanks{}

\author{Christian Richter}
\address[C. Richter]{Institute for Mathematics, Friedrich Schiller University, 07737 Jena, Germany}
\email{christian.richter@uni-jena.de}
\thanks{}

\subjclass[2010]{41A65, 46B20, 52A21}

\keywords{best coapproximation, bisector, coproximinal, gauge, Hilbert space, Minkowski space, norm}

\date{}

\begin{abstract}
We show that, for vector spaces in which distance measurement is performed using a gauge, the existence of best coapproximations in $1$-codimensional closed linear subspaces implies in dimensions $\geq 2$ that the gauge is a norm, and in dimensions $\geq 3$ that the gauge is even a Hilbert space norm.
We also show that coproximinality of all closed subspaces of a fixed dimension implies coproximinality of all subspaces of all lower finite dimensions.
\end{abstract}

\maketitle

\section{Introduction}\label{chap:intro}
A basic approximation task in Banach spaces $X$ requires to \enquote{replace} an element $y\in X$ by an element $x$ of a subset $K\subset X$ subject to some notion of proximity.
This can be implemented by the notion of best approximation, a vivid research topic whose Euclidean subcase is a common topic of undergraduate calculus or optimization courses.
In the theory of best approximation by elements of linear subspaces (or, more generally, convex subsets) $K\subset X$, see Singer's monograph \cite{Singer1970}, this notion gives rise to a set-valued operator $P_K:X\rightrightarrows X$, and one is interested in characterizations of \emph{proximinal sets} and of \emph{Chebyshev sets}, i.e., sets $K\subset X$ for which $P_K(x)$ is non-empty or a singleton for all $x\in X$, respectively.
Franchetti and Furi \cite{FranchettiFu1972} introduced a companion notion called \emph{best coapproximation} in the setting of normed spaces, yielding another set-valued operator $Q_K:X\rightrightarrows X$.
Notions of \emph{coproximinal sets} and \emph{coChebyshev sets} can be defined and investigated analogously, see \cite{AbuSirhanAl2014}.
Berens and Westphal \cite{BerensWe1980} show that a coChebyshev subset $K$ of a Hilbert space $H$ is necessarily a translate of a closed linear subspace.
Papini and Singer \cite{PapiniSi1979} link best coapproximations to Birkhoff orthogonality, provide characterizations in terms of linear functionals, give examples in normed spaces of real-valued continuous functions defined on an interval, and discuss properties of best coapproximations viewed as a set-valued operator.
(In the context of these results, also \cite[Theorem~3.2]{HorvathLaSp2015} has to be mentioned, where it is shown that in Banach spaces of dimension $\geq 3$, symmetry of the Birkhoff orthogonality relation already implies Hilbertianity.)
Franchetti and Furi \cite[Lemma~1(d) and Theorem~1]{FranchettiFu1972} show that each closed $1$-codimensional linear subspace $K$ of a Banach space $X$ of dimension $\geq 3$ is coproximinal if and only if $X$ is a Hilbert space, and that straight lines are always coproximinal in normed spaces.

The contributions of the present paper are motivated by these very results of \cite{FranchettiFu1972}, and they take place in the much broader setting of \emph{generalized Minkowski spaces}.
Those are real vector spaces $X$ equipped with a function $\gamma:X\to\RR$ (called a \emph{gauge}) which takes only non-negative values, vanishes only at $0\in X$, and meets the conditions $\gamma(\lambda x)=\lambda \gamma(x)$ and $\gamma(x+y)\leq\gamma(x)+\gamma(y)$ for all $\lambda>0$ and $x,y\in X$.
This way, every norm on $X$ is a gauge, but not vice versa.
(Norms are precisely the gauges $\gamma$ for which $\gamma(x)=\gamma(-x)$.)
Nonetheless, basic concepts from classical functional analysis can be set up in generalized Minkowski spaces in an analogous way, see Cobza{\textcommabelow{s}}'s monograph \cite{Cobzas2013}.
The number $\gamma(x-y)$ is still interpreted as a distance from $y$ to $x$, although in the gauge distance it need not coincide with the distance from $x$ to $y$ anymore.
This enables the translation of distance-based concepts of approximation theory to generalized Minkowski spaces.
Apart from functional analysis \cite{Alimov2003} and approximation theory \cite{Alimov2014b}, the study of generalized Minkowski spaces, which dates back to Minkowski \cite{Minkowski1896a}, has also attracted researchers in location science \cite{PlastriaCa2001} and convex geometry \cite{Makeev2003}.

Combining the setting of generalized Minkowski spaces and the notion of best coapproximation, we show in \cref{result:straight-line-coproximinality} that coproximinality of straight lines characterizes normed spaces among generalized Minkowski spaces in dimensions $\geq 2$, and extend the characterization of Hilbert spaces of dimension $\geq 3$ via coproximinality of $1$-codimensional linear subspaces to all generalized Minkowski spaces in \cref{result:characterization-hilbert}.
Note that we allow for infinite-dimensional vector spaces $X$ in this paper.
This was not the case in the authors' past papers on generalized Minkowski spaces such as \cite{JahnKuMaRi2015}.

\section{Background and notation}
The following definition is taken from \cite[Definition~3.10]{Jahn2019a}.
\begin{Def}
Let $(X,\gamma)$ be a generalized Minkowski space and $K\subset X$.
A point $x\in K$ is called a \emph{best coapproximation of $y\in X$ in $K$} if $\gamma(x-z)\leq\gamma(y-z)$ for all $z\in K$.
The set of all best coapproximations of $y$ in $K$ shall be denoted by $Q_K(y)$.
\end{Def}
The authors of \cite{AbuSirhanAl2014} introduce the notion of \emph{coproximinal sets}, meaning sets $K$ for which $Q_K(x)$ is non-empty for all $x\in X$.
We shall use this term also in the setting of generalized Minkowski spaces.

In \cite{Ma2000}, the \emph{bisector} $\bisec_\gamma(x,y)=\setcond{z\in X}{\gamma(z-x)=\gamma(z-y)}$ has been introduced to the setting of generalized Minkowski spaces and investigated in the case of $\dim(X)\in\setn{2,3}$.
Clearly, knowing the shape of the bisector $\bisec_{\gamma^\vee}(x,y)$ taken with respect to the gauge $\gamma^\vee:X\to\RR$ defined by $\gamma^\vee(x)\defeq \gamma(-x)$ is of interest for the study of best coapproximation.

The study of gauges is determined by the geometry of their \emph{balls}
\begin{equation*}
\mc[\gamma]{y}{\lambda}\defeq\setcond{x\in X}{\gamma(x-y)\leq \lambda}, \quad
\mo[\gamma]{y}{\lambda}\defeq\setcond{x\in X}{\gamma(x-y)<\lambda},
\end{equation*}
and spheres
\begin{equation*}
\ms[\gamma]{y}{\lambda}\defeq\setcond{x\in X}{\gamma(x-y)= \lambda}.
\end{equation*}
Balls are convex sets, i.e., sets which contain the whole \emph{line segment} $[x,y]\defeq\setcond{x+\lambda (y-x)}{0\leq \lambda\leq 1}$ with each two points $x$ and $y$ they contain.
We shall use the notations $\strline{x}{y}$ for the \emph{straight line} passing through points $x,y\in X$, $\clray{x}{y}\defeq\setcond{x+\lambda(y-x)}{\lambda\geq 0}$ for the \emph{ray} starting at $x$ and passing through $y$, $\cone{A}\defeq\setcond{\lambda x}{x\in A,\lambda >0}$ for the \emph{conical hull} of a set $A\subset X$, and $\lin{A}$ for its \emph{linear hull}.
For $A_1,A_2\subset X$ and $x\in X$, we write $A_1\pm A_2\defeq\setcond{a\pm b}{a\in A_1,b\in A_2}$ and $x+A_1=A_1+x\defeq\setn{x}+A_1$.
A linear subspace $H$ of $X$ is said to be \emph{$1$-codimensional} if there exists a vector $v\in X\setminus H$ such that $X=\setcond{x+\alpha v}{x\in H,\alpha\in\RR}$.

Using the notion of balls, one may write for the set of best coapproximations
\begin{equation*}
Q_K(y)=K\cap\bigcap_{z\in K}\mc[\gamma]{z}{\gamma(y-z)}=K\cap\bigcap_{z\in K, \lambda>0:\,y\in\mc[\gamma]{z}{\lambda}}\mc[\gamma]{z}{\lambda}.
\end{equation*}

\section{Characterizing normed spaces in terms of coproximinality}
Coproximinality of all straight lines characterizes $2$-dimensional normed spaces among the $2$-dimensional generalized Minkowski spaces.
\begin{Lem}\label{result:straight-line-dimension-2}
Let $(X,\gamma)$ be a generalized Minkowski space with $\dim(X)=2$.
The following statements are equivalent.
\begin{enumerate}[label={(\roman*)},leftmargin=*,align=left,noitemsep]
\item{The gauge $\gamma$ is a norm.\label{norm}}
\item{Every straight line $K\subset X$ is coproximinal.\label{straight-line-coproximinality}}
\end{enumerate}
\end{Lem}

\begin{proof}
For the implication \ref{norm}$\Rightarrow$\ref{straight-line-coproximinality}, see \cite[Lemma~1(d)]{FranchettiFu1972}.
For the converse implication, assume that $\gamma$ is not a norm.
By \cite[4.1]{Soltan2005}, there exists a chord $[x_0,x_1]$ of $\mc[\gamma]{0}{1}$ such that $0\in [x_0,x_1]$ and $[x_0,x_1]$ is no affine diameter of $\mc[\gamma]{0}{1}$, i.e., $\mc[\gamma]{0}{1}$ has no pair of parallel supporting lines passing through $x_0$ and $x_1$, respectively.
Thus there exists another chord $[y_0,y_1]$ of $\mc[\gamma]{0}{1}$ and a number $\lambda>1$ such that $y_1-y_0=\lambda(x_1-x_0)$.

Let $K=\strline{x_0}{x_1}$.
Then $0,y_1-y_0\in K$ and $\gamma(y_1-0)=1=\gamma(y_1-(y_1-y_0))$.
It follows that
\begin{align*}
Q_K(y_1)&=K\cap\bigcap_{z\in K}\mc[\gamma]{z}{\gamma(y_1-z)}\\
&\subset K\cap \lr{\mc[\gamma]{0}{1}\cap\mc[\gamma]{y_1-y_0}{1}}\\
&=(K\cap \mc[\gamma]{0}{1})\cap(K\cap \mc[\gamma]{y_1-y_0}{1})\\
&=[x_0,x_1]\cap ([x_0,x_1]+y_1-y_0)\\
&=\emptyset
\end{align*}
Therefore, the straight line $K$ is not coproximinal.
\end{proof}

Let us give two alternative proofs of the implication \ref{norm}$\Rightarrow$\ref{straight-line-coproximinality} in \cref{result:straight-line-dimension-2}.
\begin{proof}[First proof.]
Let $y\in X$.
Let $u_1$ and $u_2$ be the points of intersection of the straight line $y+(K-K)$ and $\ms[\gamma]{y}{1}$.
The convex body $\mc[\gamma]{y}{1}$ admits a pair of parallel supporting lines $L_1$ at $u_1$ and $L_2$ at $u_2$.
The straight lines $L_1$ and $L_2$ intersect $K$ in points $v_1$ and $v_2$, respectively.
Take $x=\frac{1}{2}(v_1+v_2)$. 

Case 1: $L_1\cap \ms[\gamma]{y}{1}=\setn{u_1}$.
Then \cite[Lemma~2.1.1.1]{Ma2000} implies that the bisector $\bisec_\gamma(x,y)$ is contained in the interior of the convex hull of the union of $K$ and $y+(K-K)$ or, in other words, $\gamma(y-z)>\gamma(x-z)$ for all $z\in K$, see \cref{fig:case-1} for an illustration.

\begin{figure}
\begin{center}
\begin{tikzpicture}[line cap=round,line join=round,>=stealth,x=1.0cm,y=0.5cm]
\clip(-4.4,-3.7) rectangle (5,3.7);
\draw [domain=-4.4:5] plot(\x,{(-2.-1.*\x)/0.34392145520964856});
\draw [domain=-4.:5] plot(\x,{(--2.-1.*\x)/0.34392145520964856});
\draw (-4,1)--(4,1) node[right]{$L_1$};
\draw (-4,-1)--(4,-1) node[right]{$L_2$};
\draw [fill=black] (-2,0) circle (1.5pt) node[left]{$y$};
\draw [fill=black] (2,0) circle (1.5pt) node[right]{$x$};
\draw [fill=black] (-2.3439214552096486,1) circle (1.5pt) node[above right]{$u_1$};
\draw [fill=black] (-1.6560785447903514,-1) circle (1.5pt) node[below left]{$u_2$};
\draw [fill=black] (2.3439214552096486,-1) circle (1.5pt) node[below left]{$v_2$};
\draw [fill=black] (1.6560785447903514,1) circle (1.5pt) node[above right]{$v_1$};
\fill[color=white] (-0.8,3.4) circle(5pt);
\fill[color=white]  (0.8,-3.4) circle(5pt);
\fill (3.5,-3.4) node{$K$};
\end{tikzpicture}
\end{center}
\caption{Notation in Case 1 of an alternative proof of the implication \ref{norm}$\Rightarrow$\ref{straight-line-coproximinality} in \cref{result:straight-line-dimension-2}.\label{fig:case-1}}
\end{figure}
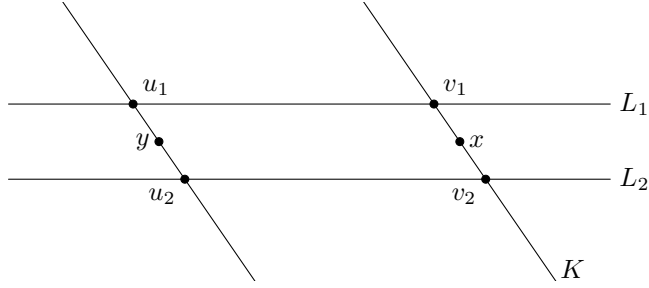

Case 2: $L_1\cap \ms[\gamma]{y}{1}=[w_1,w_2]$.
Assume that there exists $\lambda>0$ such that $w_2-w_1=\lambda(x-y)$.
(Else, change the roles of $w_1$ and $w_2$.)
Let $s_1$ be the intersection point of $\clray{y}{w_2}$ and $\clray{x}{w_1+x-y}$, and $s_2$ be the intersection point of $\clray{y}{2y-w_1}$ and $\clray{x}{2x-(w_2+x-y)}$.
By \cite[Theorem~2.1 and Proposition 2.4]{JahnSp2015}, the bisector $\bisec_\gamma(x,y)$ can be written as a union $A_1\cup A_2\cup A_3$ with $A_1=(\cone{[w_1,w_2]-y}+y)\cap (\cone{[w_1,w_2]-y}+x)$, $A_2=(-\cone{[w_1,w_2]-y}+y)\cap (-\cone{[w_1,w_2]-y}+x)$, and $A_3$ being contained in the convex hull of $\setn{x,y,s_1,s_2}$, see \cref{fig:case-2} for an illustration.
But then, we also have $\gamma(y-z)\geq\gamma(x-z)$ for all $z\in K$.

\begin{figure}
\begin{center}
\begin{tikzpicture}[line cap=round,line join=round,>=stealth,x=1.0cm,y=0.5cm]
\clip(-4.4,-3.7) rectangle (5,3.7);
\draw [domain=-2.0:5] plot(\x,{(--2.--1.*\x)/0.5});
\draw [domain=-4.4:2.0] plot(\x,{(-2.--1.*\x)/-1.});
\draw [domain=-2.0:5] plot(\x,{(-2.-1.*\x)/1.});
\draw [domain=-4.4:2.0] plot(\x,{(--2.-1.*\x)/-0.5});
\draw[line width=1.5pt] (-3,1)--(-1.5,1);
\draw[line width=1.5pt] (-2.5,-1)--(-1,-1);
\draw[line width=1.5pt] (1.5,-1.)--(3,-1);
\draw[line width=1.5pt] (2.5,1)--(1,1);
\draw [domain=-4.4:5] plot(\x,{(-2.-1.*\x)/0.34392145520964856});
\draw [domain=-4.:5] plot(\x,{(--2.-1.*\x)/0.34392145520964856});
\draw (-4,1)--(4,1) node[right]{$L_1$};
\draw (-4,-1)--(4,-1) node[right]{$L_2$};
\draw [fill=black] (-2,0) circle (1.5pt) node[left]{$y$};
\draw [fill=black] (2,0) circle (1.5pt) node[right]{$x$};
\draw [fill=black] (2.5,1) circle (1.5pt);
\draw [fill=black] (1,1) circle (1.5pt);
\draw [fill=black] (1.5,-1) circle (1.5pt);
\draw [fill=black] (3,-1) circle (1.5pt);
\draw [fill=black] (-3,1) circle (1.5pt) node[below]{$w_1$};
\draw [fill=black] (-1.5,1) circle (1.5pt) node[below right]{$w_2$};
\draw [fill=black] (-2.5,-1) circle (1.5pt);
\draw [fill=black] (-1,-1) circle (1.5pt);
\draw [fill=black] (-0.6666666666666666,2.6666666666666665) circle (1.5pt) node[left]{$s_1$};
\draw [fill=black] (0.6666666666666666,-2.6666666666666665) circle (1.5pt) node[left]{$s_2$};
\draw [fill=black] (-2.3439214552096486,1) circle (1.5pt) node[above right]{$u_1$};
\draw [fill=black] (-1.6560785447903514,-1) circle (1.5pt) node[below left]{$u_2$};
\draw [fill=black] (2.3439214552096486,-1) circle (1.5pt) node[below left]{$v_2$};
\draw [fill=black] (1.6560785447903514,1) circle (1.5pt) node[above right]{$v_1$};
\fill[pattern=dots] (-0.6772543773021731,-5.354508754604346) -- (2.8249445219275753,-4.824944521927575) -- (0.6666666666666666,-2.6666666666666665) -- cycle;
\fill[pattern=dots] (-0.6666666666666666,2.6666666666666665) -- (0.5320710146544925,5.064142029308985) -- (-3.209905725590878,5.209905725590878) -- cycle;
\fill[color=white] (-0.8,3.4) circle(5pt);
\fill (-0.8,3.4) node{$A_1$};
\fill[color=white]  (0.8,-3.4) circle(5pt);
\fill (0.8,-3.4) node{$A_2$};
\fill (3.5,-3.4) node{$K$};
\end{tikzpicture}
\end{center}
\caption{Notation in Case 2 of an alternative proof of the implication \ref{norm}$\Rightarrow$\ref{straight-line-coproximinality} in \cref{result:straight-line-dimension-2}\label{fig:case-2}.}
\end{figure}
In either case, $x \in Q_K(y)$.
\end{proof}

\begin{proof}[Second proof.]
Let $K\subset X$ be a straight line and let $y_0\in X$.
Coproximinality is translation invariant, i.e., if a set $K\subset X$ is coproximinal, then so is $K+x$ for every $x\in X$.
Thus there is no loss of generality in assuming that $K$ is a linear subspace of $X$.
Then $K\cap\mc[\gamma]{0}{1}=[-x_0,x_0]$, and there is a linear functional $f:X\to\RR$, $f\neq 0$, such that $\max\setcond{f(x)}{x\in\mc[\gamma]{0}{1}}=f(x_0)\eqdef \alpha>0$ and $\min\setcond{f(x)}{x\in\mc[\gamma]{0}{1}}=f(-x_0)=-\alpha<0$.

For $z\in K$ and $\lambda\geq 0$, we then have $\mc[\gamma]{z}{\lambda}\cap K=f^{-1}([f(z)-\lambda \alpha,f(z)+\lambda \alpha])\cap K$, and $y_0\in\mc[\gamma]{z}{\lambda}$ implies $f(y_0)\in [f(z)-\lambda\alpha,f(z)+\lambda\alpha]$.
Hence,
\begin{align*}
Q_K(y_0)&=K\cap \bigcap_{z\in K,\lambda>0:\, y_0\in\mc[\gamma]{z}{\lambda}}\mc[\gamma]{z}{\lambda}\\
&=\bigcap_{z\in K,\lambda>0:\, y_0\in\mc[\gamma]{z}{\lambda}}\lr{f^{-1}([f(z)-\lambda \alpha,f(z)+\lambda\alpha])\cap K}\\
&\supset \bigcap_{z\in K,\lambda>0:\, y_0\in\mc[\gamma]{z}{\lambda}}\lr{f^{-1}(f(y_0))\cap K}\\
&=f^{-1}(f(y_0))\cap K\\
&\neq \emptyset.
\end{align*}
For the latter inequality, note that the set $f^{-1}(f(y_0))$ is a straight line parallel to supporting lines of $\mc[\gamma]{0}{1}$ at $x_0$ and $-x_0$, so it cannot be disjoint from (i.e., parallel to) $\strline{-x_0}{x_0}=K$.
\end{proof}

The characterization established in \cref{result:straight-line-dimension-2} translates directly to higher dimensions.
\begin{Satz}\label{result:straight-line-coproximinality}
Let $(X,\gamma)$ be a generalized Minkowski space with $\dim(X)\geq 2$.
The following statements are equivalent.
\begin{enumerate}[label={(\roman*)},leftmargin=*,align=left,noitemsep]
\item{The gauge $\gamma$ is a norm.}
\item{Every straight line $K\subset X$ is coproximinal.}
\end{enumerate}
\end{Satz}
\begin{proof}
We reduce the problem to the $2$-dimensional case and apply \cref{result:straight-line-dimension-2}.
First, by translational invariance of coproximinality, it is sufficient to consider $1$-dimensional linear subspaces instead of arbitrary straight lines.
Second, for the coproximinality of a set $K\subset X$, the points of $K$ are not crucial as we always have $Q_K(y)=\setn{y}$ when $y\in K$.
Third, for a set $K\subset X$ and a point $y\in X$, the set $Q_K(y)$ stays the same when viewed in the generalized Minkowski space $(Y,\gamma\vert_Y)$ instead of $(X,\gamma)$, with $Y\defeq\lin{K\cup\setn{y}}$ and $\gamma\vert_Y$ being the restriction of $\gamma$ to $Y$.
Taking these observations into account, the coproximinality of every straight line $K\subset X$ is equivalent to the coproximinality of every $1$-dimensional linear subspace $K$ of $X$ with respect to every $2$-dimensional linear subspace $(Y,\gamma\vert_Y)$ of $(X,\gamma)$ with $K \subseteq Y$.
\end{proof}


\section{Some remarks on topology}

In order to give an account on the topological arguments that come with the following characterization of Hilbert spaces, we loosely follow the exposition of \cite{Cobzas2013}.

For a generalized Minkowski space $(X,\gamma)$, we call a set $V\subset X$ a \emph{neighborhood} of $x\in X$, if there exists a number $\lambda>0$ such that $\mo[\gamma]{x}{\lambda}\subset V$.
This way, the family of neighborhoods of a fixed point forms a filter of subsets of $X$. Clearly, another basis of that filter is given by $\mc[\gamma]{x}{\lambda}$, $\lambda > 0$.
Those subsets of $X$ which are neighborhoods of all of their points form the family of all open sets of a topology $\tau_\gamma$ on $X$.
A set $F \subset X$ is closed in that topology if $X \setminus F$ is open or, equivalently, if $F$ is closed under $\tau_\gamma$-limits.
Here the expression $\lim_{n \to \infty} x_n=x_0$ is defined by $\lim_{n \to \infty} \gamma(x_n-x_0)=0$.

Note that $\lim_{n \to \infty} \gamma(x_n-x_0)=0$ does not imply $\lim_{n \to \infty} \gamma(x_0-x_n)=0$ (see \cref{result:example_B_not_closed} below). That is, the gauges $\gamma$ and $\gamma^\vee$ may give rise to different concepts of convergence. In fact, two gauges generate the same topology or, equivalently, the same concepts of convergence if and only if they are equivalent in the following sense.
\begin{Prop}
\label{result:coinciding-topology-gauge-equivalence}
Let $(X,\gamma_1)$ and $(X,\gamma_2)$ be two generalized Minkowski spaces over the same vector space $X$.
The following statements are equivalent.
\begin{enumerate}[label={(\roman*)},leftmargin=*,align=left,noitemsep]
\item{The gauges $\gamma_1$ and $\gamma_2$ generate the same topology, i.e., we have $\tau_{\gamma_1}=\tau_{\gamma_2}$.\label{coinciding-topology}}
\item{For all $(x_n)_{n=0}^\infty \subset X$, $\lim_{n \to \infty} x_n=x_0$ w.r.t.\ $\gamma_1$ if and only if $\lim_{n \to \infty} x_n=x_0$ w.r.t.\ $\gamma_2$.\label{convergence-equivalence}}
\item{The gauges $\gamma_1$ and $\gamma_2$ are \emph{equivalent}, i.e., there exist numbers $c_0,c_1>0$ such that $c_0\gamma_1(x)\leq\gamma_2(x)\leq c_1\gamma_1(x)$ for all $x\in X$.\label{gauge-equivalence}}
\end{enumerate}
\end{Prop}
\begin{proof}
First assume that \ref{coinciding-topology} holds true.
As $\mc[\gamma_1]{0}{1}$ is a $\tau_{\gamma_1}$-neighborhood of $0$, it is also a $\tau_{\gamma_2}$-neighborhood of $0$, so there exists a number $\lambda>0$ such that $\mc[\gamma_2]{0}{\lambda}\subset\mc[\gamma_1]{0}{1}$.
Thus, for $x\neq 0$, we have $\frac{\lambda}{\gamma_2(x)}x\in\mc[\gamma_2]{0}{\lambda}\subset\mc[\gamma_1]{0}{1}$, i.e., $\lambda\gamma_1(x)\leq \gamma_2(x)$.
This means we can take $c_0=\lambda$ in \ref{gauge-equivalence}.
Interchanging the roles of $\gamma_1$ and $\gamma_2$ gives a similar expression for $c_1$ in \ref{gauge-equivalence}.

For \ref{gauge-equivalence}$\Rightarrow$\ref{convergence-equivalence}, simply note that \ref{gauge-equivalence} yields the equivalence of $\lim_{n \to \infty} \gamma_1(x_n-x_0)=0$ with $\lim_{n \to \infty} \gamma_2(x_n-x_0)=0$.

The implication \ref{convergence-equivalence}$\Rightarrow$\ref{coinciding-topology} is obvious, since closed sets are characterized by closedness under limits of sequences.
\end{proof}

The following result gives a characterization when the topology generated by a gauge can be generated by a norm and when it is a vector space topology, see also \cite{BachirFl2020} for a related discussion in a broader setting.
\begin{Prop}
\label{result:topology-equivalent-gauges}
Let $(X,\gamma)$ be a generalized Minkowski space.
The following statements are equivalent.
\begin{enumerate}[label={(\roman*)},leftmargin=*,align=left,noitemsep]
\item{The topology $\tau_\gamma$ is a vector space topology, i.e., the mappings $X\times X\to X$, $(x,y)\mapsto x+y$ and $\RR\times X\to X$, $(\lambda,x)\mapsto \lambda x$ are continuous w.r.t.\ the product topologies.\label{vector-space-topology}}
\item{The mapping $X\to X$, $x\mapsto -x$, is continuous at $0$ w.r.t.\ $\tau_\gamma$.\label{negative-continuity}}
\item{The gauges $\gamma$ and $\gamma^\vee$ are equivalent.\label{opposite-equivalence}}
\item{The gauge $\gamma$ and the norm $\mnorm{x}_\gamma=\max\setn{\gamma,\gamma^\vee}$ are equivalent.\label{max-equivalence}}
\item{There exists a norm $\mnorm{\cdot}:X\to\RR$ such that $\gamma$ and $\mnorm{\cdot}$ are equivalent.\label{norm-equivalence}}
\item{There exists a norm $\mnorm{\cdot}:X\to\RR$ such that $\tau_\gamma=\tau_{\mnorm{\cdot}}$.\label{norm-topology}}
\end{enumerate}
\end{Prop}

\begin{proof}
The implication \ref{vector-space-topology}$\Rightarrow$\ref{negative-continuity} is trivial.
For \ref{negative-continuity}$\Rightarrow$\ref{opposite-equivalence}, according to the assumption there exists $\lambda > 0$ such that $\mc[\gamma]{0}{\lambda}$ is mapped into the neighborhood $\mc[\gamma]{0}{1}$ of $-0=0$. That is, $-\mc[\gamma]{0}{\lambda} \subset \mc[\gamma]{0}{1}$, whence also $\mc[\gamma]{0}{\lambda} \subset -\mc[\gamma]{0}{1}$. These give $\mc[\gamma^\vee]{0}{\lambda} \subset \mc[\gamma]{0}{1}$, $\mc[\gamma]{0}{\lambda} \subset \mc[\gamma^\vee]{0}{1}$ and in turn $\lambda\gamma(x) \leq \gamma^\vee(x) \leq \frac{1}{\lambda} \gamma(x)$.

Now let us show \ref{opposite-equivalence}$\Rightarrow$\ref{max-equivalence}.
By assumption, there exist numbers $c_0,c_1>0$ such that $c_0\gamma(x)\leq\gamma^\vee(x)\leq c_1\gamma(x)$.
But then also 
\begin{equation*}
\gamma(x)\leq\max\setn{\gamma(x),\gamma^\vee(x)}\leq\max\setn{\gamma(x),c_1\gamma(x)}=\max\setn{1,c_1}\gamma(x)
\end{equation*}
for arbitrary $x\in X$.

The implication \ref{max-equivalence}$\Rightarrow$\ref{norm-equivalence} is again trivial, \ref{norm-equivalence}$\Rightarrow$\ref{norm-topology} follows from \cref{result:coinciding-topology-gauge-equivalence}, and \ref{norm-topology}$\Rightarrow$\ref{vector-space-topology} is trivial, too.
\end{proof}

As a consequence of \ref{opposite-equivalence}, balls $\mc[\gamma]{x}{\lambda}$ with $x\in X$ and $\lambda>0$ are closed in $\tau_\gamma$.
Indeed, for $n \in \mathbb{N}$, let $x_n \in \mc[\gamma]{x}{\lambda}$ and assume that there exists $x_0 \in X$ such that $\lim_{n \to \infty} x_n=x_0$ w.r.t.\ $\gamma$.
Then, by \ref{opposite-equivalence}, also $\lim_{n \to \infty} x_n=x_0$ w.r.t.\ $\gamma^\vee$ and $\gamma(x_0-x) \leq \gamma(x_0-x_n)+\gamma(x_n-x) \leq \gamma^\vee(x_n-x_0)+\lambda \stackrel{n \to \infty}{\longrightarrow} \lambda$. Hence $x_0 \in \mc[\gamma]{x}{\lambda}$.

Note that for general gauges $\gamma$, $\mc[\gamma]{0}{1}$ need not be closed in $\tau_\gamma$.
This is stated in \cite[Proposition~1.1.8.1]{Cobzas2013}, but there the concept of a gauge is more flexible, since it requires $x=0$ if and only if $\gamma(x)=0$ and $\gamma(-x)=0$, whereas we suppose $x=0$ if and only if $\gamma(x)=0$.
Therefore we add the following example that meets our stronger concept of a gauge.

\begin{Bsp}
\label{result:example_B_not_closed}
Let the space $X=\ell_1$ of all absolutely convergent sequences of reals be equipped with the gauge $\gamma(\xi_1,\xi_2,\ldots)=\max\setn{\sup_{i \geq 1}|\xi_i|,\sum_{i=1}^\infty \xi_i}$. Then the sequence $(x_n)_{n=1}^\infty \subset \ell_1$ defined by $x_n=\bigg(\underbrace{\frac{1}{n},\ldots,\frac{1}{n}}_{n \text{ times}},0,0,\ldots\bigg)$ satisfies $\gamma(x_n) \equiv 1$ and $\lim_{n \to \infty} (-x_n)=0$. Moreover, $\mc[\gamma]{0}{1}$ is not closed in $\tau_\gamma$.
\end{Bsp}

\begin{proof}
The claims $\gamma(x_n) = 1$ and $\lim_{n \to \infty} (-x_n)=0$ are obvious.

To see that $\mc[\gamma]{0}{1}$ is not closed, we consider $x_0=\left(1,\frac{1}{2},\frac{1}{4},\frac{1}{8},\ldots\right) \in \ell_1$. Then $\gamma(x_0)=2$, whence $x_0 \notin \mc[\gamma]{0}{1}$. 
The sequence $(x_0-x_n)_{n=1}^\infty \subset \ell_1$ satisfies $(x_0-x_n)_{n=1}^\infty \subset \mc[\gamma]{0}{1}$, because
\begin{align*}
\gamma(x_0-x_n)&=\max\setn{\sup\setn{\abs{1-\frac{1}{n}},\abs{\frac{1}{2}-\frac{1}{n}},\ldots,\abs{\frac{1}{2^{n-1}}-\frac{1}{n}},\frac{1}{2^n}}, 2-1} = 1,
\end{align*}
and $\lim_{n \to \infty} (x_0-x_n)=x_0$, since $\lim_{n \to \infty} \gamma((x_0-x_n)-x_0)=\lim_{n \to \infty} \gamma(-x_n)=0$. Having found a sequence $(x_0-x_n)_{n=1}^\infty \subset \mc[\gamma]{0}{1}$ such that $\lim_{n \to \infty} (x_0-x_n)=x_0 \notin \mc[\gamma]{0}{1}$, we see that $\mc[\gamma]{0}{1}$ is not closed. 
\end{proof}

Taking into account that a gauge can behave strangely compared with a norm, we stress here that the local topological behaviors of gauges and norms agree in the following sense.

\begin{Lem}
\label{result:finite_dimensional}
Let $(X,\gamma)$ be a generalized Minkowski space and let $X_0$ be a finite-dimensional linear subspace of $X$.
\begin{enumerate}[label={(\roman*)},leftmargin=*,align=left,noitemsep]
\item{The topology on $X_0$ generated by $\gamma$ agrees with the canonical (Euclidean) vector space topology on $X_0$.\label{Euclidean_topology}}
\item{The subspace $X_0$ is closed in $\tau_\gamma$.\label{closed_subspaces}}
\end{enumerate}
\end{Lem}

\begin{proof}
For \ref{Euclidean_topology}, see \cite[Theorem~9]{GarciaRaffi2005}.
For \ref{closed_subspaces}, consider a sequence $(x_n)_{n=1}^\infty \subseteq X_0$ such that $\lim_{n \to \infty} x_n=x_0 \in X$. Application of \ref{Euclidean_topology} to the finite-dimensional subspace $X_1=\lin{X_0 \cup\setn{x_0}}$ shows that $x_0 \in X_0$. So $X_0$ is closed.
\end{proof}

Claim \ref{Euclidean_topology} and \cref{result:coinciding-topology-gauge-equivalence} imply in particular that all gauges on a finite-dimensional linear space are equivalent and give rise to the same concepts of boundedness and convergence. 


\section{Characterizing Hilbert spaces in terms of coproximinality}

Here we want to show the following characterization of Hilbert spaces.

\begin{Satz}\label{result:characterization-hilbert}
Let $(X,\gamma)$ be a generalized Minkowski space of dimension $\dim(X)\geq 3$. 
The following statements are equivalent.
\begin{enumerate}[label={(\roman*)},leftmargin=*,align=left,noitemsep]
\item{The generalized Minkowski space $(X,\gamma)$ is a Hilbert space, i.e., $\gamma$ is a norm induced by an inner product and $(X,\gamma)$ is complete w.r.t.\ that norm.\label{Hilbert-original}}
\item{Every closed $1$-codimensional linear subspace $K$ of $X$ is coproximinal.\label{Hilbert-coproximinal}}
\end{enumerate}
\end{Satz}

We recall that the above is established in \cite[Theorem 1]{FranchettiFu1972} under the additional assumptions that $\gamma$ is a norm and that the normed space $(X,\gamma)$ is complete.


\subsection{The necessity of an inner product structure}

The coproximinality of all linear subspaces of a fixed dimension implies the coproximinality of all linear subspaces of lower finite dimension.
\begin{Prop}\label{result:intermediate-subspaces}
Let $(X,\gamma)$ be a generalized Minkowski space.
If there exists a finite-dimensional linear subspace $X_0\subset X$ which is not coproximinal, then there also exists a closed $1$-codimensional linear subspace $H\subset X$ with $X_0\subset H$ such that $H$ as well as all subspaces $X_1$ with $X_0\subset X_1\subset H$ are not coproximinal.
\end{Prop}

\begin{proof}
As $X_0$ is not coproximinal, there exists a point $y_0\in X$ such that
\begin{equation}
Q_{X_0}(y_0)=X_0\cap\underbrace{\bigcap_{z\in X_0}\mc[\gamma]{z}{\gamma(y_0-z)}}_{\eqdef C}=\emptyset.\label{eq:not-coproximinal}
\end{equation}
Then there exists some $n_0 \in \setn{1,2,\ldots}$ such that
\begin{equation}
\label{eq:open-intersection}
X_0 \cap \underbrace{\bigcup_{x \in C} \mo[\gamma]{x}{\frac{1}{n_0}}}_{\eqdef\;G(n_0)} = \emptyset.
\end{equation}
Indeed, suppose, contrary to \eqref{eq:open-intersection}, that there are $x_n \in X_0 \cap G(n)$ for $n=1,2,\ldots$
We fix some $z_0 \in X_0$ and obtain 
\begin{equation*}
(x_n)_{n=1}^\infty \subset X_0 \cap G(1) \subset X_0 \cap \bigcup_{x \in \mc[\gamma]{z_0}{\gamma(y_0-z_0)}}\mo[\gamma]{x}{1} \subset X_0 \cap \mo[\gamma]{z_0}{\gamma(y_0-z_0)+1}.
\end{equation*}
So $(x_n)_{n=1}^\infty$ is bounded in the (closed) finite-dimensional subspace $X_0$ (cf.\ \cref{result:finite_dimensional}), thus having a convergent subsequence. W.l.o.g., $(x_n)_{n=1}^\infty$ converges itself, i.e., $\lim_{n \to\infty} \gamma(x_n-x_0)=0$ for some $x_0 \in X_0$.
By \cref{result:finite_dimensional}, also $\lim_{n \to\infty} \gamma(x_0-x_n)=0$. Since $x_n \in G(n)$, there exist $\tilde{x}_n \in C$ such that $\gamma(x_n-\tilde{x}_n)< \frac{1}{n}$. Then every $z \in X_0$ satisfies
\begin{equation*}
\gamma(x_0-z) \leq \gamma(x_0-x_n)+\gamma(x_n-\tilde{x}_n)+\gamma(\tilde{x}_n-z) < \gamma(x_0-x_n)+\frac{1}{n}+\gamma(y_0-z)
\end{equation*}
for all $n=1,2,\ldots$, whence $x_0 \in \mc[\gamma]{z}{\gamma(y_0-z)}$.
Therefore $x_0 \in C$ in addition to $x_0 \in X_0$. This contradicts \eqref{eq:not-coproximinal} and proves \eqref{eq:open-intersection}. 

Now Theorem~2.2.8 from \cite{Cobzas2013} says that there is a continuous linear functional $h$ from $(X,\gamma)$ into $\RR$ that separates the open convex set $G(n_0)$ from the disjoint linear subspace $X_0$ in the sense that 
\begin{equation}
\label{eq:separation}
h(x) < h(z)=0 \quad\text{ for all }\quad x \in G(n_0),\, z \in X_0.
\end{equation}
The null space $H\defeq h^{-1}(0)$ is a closed $1$-codimensional subspace of $X$.

Let $X_1$ be a linear subspace of $X$ such that $X_0\subset X_1\subset H$.
Then we have
\begin{align*}
Q_{X_1}(y_0)&=X_1\cap\bigcap_{z\in X_1}\mc[\gamma]{z}{\gamma(y_0-z)}\subset H\cap\bigcap_{z\in X_0}\mc[\gamma]{z}{\gamma(y_0-z)}\\
&=H\cap C \stackrel{\eqref{eq:open-intersection}}{\subset} H \cap G(n_0) \stackrel{\eqref{eq:separation}}{=}\emptyset,
\end{align*}
which means that $X_1$ is not coproximinal.
\end{proof}

\begin{Kor}
\label{result:necessity-inner-product}
Let $(X,\gamma)$ be a generalized Minkowski space of dimension $\dim(X)\geq 3$. If every closed $1$-codimensional linear subspace $K$ of $X$ is coproximinal, then $\gamma$ is a norm induced by an inner product on $X$.
\end{Kor}

\begin{proof}
By \cite[Section~4]{JordanNe1935}, it is sufficient to show that the restriction of $\gamma$ to every $2$-dimensional linear subspace of $X$ is a norm induced by an inner product.
To this end, fix an arbitrary $3$-dimensional subspace $X_0$ of $X$.
By \cref{result:intermediate-subspaces}, all $1$-dimensional or $2$-dimensional linear subspaces of $X$ are coproximinal.
In particular, all $1$-dimensional or $2$-dimensional linear subspaces of $X_0$ are coproximinal, also when viewed in $(X_0,\gamma\vert_{X_0})$.
By \cref{result:straight-line-coproximinality}, the generalized Minkowski space $(X_0,\gamma\vert_{X_0})$ is in fact a normed space (of dimension $3$), and since each of its $2$-dimensional linear subspaces is coproximinal, \cite[Theorem~1]{FranchettiFu1972} shows that $\gamma\vert_{X_0}$ is a norm induced by an inner product.
As every $2$-dimensional linear subspace of $X$ is a linear subspace of some choice of $X_0$, we have shown that the restriction of $\gamma$ to every $2$-dimensional linear subspace of $X$ is in fact a norm induced by an inner product.
\end{proof}


\subsection{The necessity of completeness}

First we note that in the setting of an inner product space points of best coapproximation within a linear subspace are found by the orthogonal projection onto that space, see also \cite[Section 2]{FranchettiFu1972}.

\begin{Lem}
\label{result:projection}
Let $K$ be a linear subspace of an inner product space $(X,\skpr{\cdot}{\cdot})$, let $y_0 \in X$ and let $z_0 \in K$. Then $z_0 \in Q_K(y_0)$ if and only if $y_0-z_0 \perp K$, i.e., $\skpr{y_0-z_0}{z}=0$ for all $z \in K$.
\end{Lem}

\begin{proof}
We reformulate the claim $z_0 \in Q_K(y_0)$ successively:

\begin{align*}
&\Leftrightarrow&&\mnorm{z_0-(z_0+\lambda z)} \leq \mnorm{y_0-(z_0+\lambda z)}\text{ for all }z \in K,\lambda \in \RR\\
&\Leftrightarrow&&\mnorm{\lambda z}^2 \leq \mnorm{(y_0-z_0)-\lambda z}^2\text{ for all }z \in K,\lambda \in \RR\\
&\Leftrightarrow&&\mnorm{\lambda z}^2 \leq \mnorm{y_0-z_0}^2-2\skpr{y_0-z_0}{\lambda z}+\mnorm{\lambda z}^2\text{ for all }z \in K,\lambda \in \RR\\
&\Leftrightarrow&&2\lambda\skpr{y_0-z_0}{z}\leq\mnorm{y_0-z_0}^2\text{ for all }z \in K,\lambda \in \RR\\
&\Leftrightarrow&&\skpr{y_0-z_0}{z}=0\text{ for all }z \in K
\end{align*}
\end{proof}

Now the necessity of completeness is obtained via the Riesz representation theorem of continuous linear functionals.

\begin{Lem}
\label{result:necessity-complete}
The following are equivalent for every inner product space $(X,\skpr{\cdot}{\cdot})$.
\begin{enumerate}[label={(\roman*)},leftmargin=*,align=left,noitemsep]
\item{The vector space $X$ is complete w.r.t.\ the norm $\mnorm{\cdot}$ induced by $\skpr{\cdot}{\cdot}$.\label{complete}}
\item{For every continuous linear functional $f:X\to\RR$, there exists $x_0 \in X$ such that $f(x)=\skpr{x}{x_0}$ for all $x \in X$ (Riesz representation theorem).\label{Riesz}}
\item{Every closed $1$-codimensional linear subspace $K$ of $X$ is coproximinal.\label{coproximinal}}
\end{enumerate}
\end{Lem}

\begin{proof}
For \ref{complete}$\Leftrightarrow$\ref{Riesz}, see \cite[Theorem 3.3.5]{Istratescu1987}. For \ref{complete}$\Rightarrow$\ref{coproximinal}, we use the fact that the orthogonal projection $P_K:X\to K$ is well-defined by \ref{complete} (cf.\ \cite[Theorem 3.3.5 and Definition 3.3.9]{Istratescu1987}), and that $P_K$ maps every $y \in X$ onto $P_K(y) \in Q_K(y)$ by \cref{result:projection}. 

For \ref{coproximinal}$\Rightarrow$\ref{Riesz}, let $f:X\to\RR$ be a continuous linear functional.
We can assume that $f \neq 0$.
Then the null space $K\defeq f^{-1}(0)$ is a closed $1$-codimensional subspace of $X$. We pick $y_0 \in X \setminus K$.
By \ref{coproximinal}, there is $z_0 \in Q_K(y_0)$. \cref{result:projection} yields $y_0-z_0 \perp K$.
We put $x_0\defeq\frac{f(y_0)}{\mnorm{y_0-z_0}^2}(y_0-z_0)$.
For every $x \in X$, there exist $z \in K$ and $\lambda \in \RR$ such that $x=z+\lambda(y_0-z_0)$, since $X=K \oplus \lin{\setn{y_0-z_0}}$.
Thus
\begin{align*}
f(x)&=f(z+\lambda(y_0-z_0))\\
&=f(z-\lambda z_0)+\lambda f(y_0)\\
&=0+\skpr{\lambda(y_0-z_0)}{\frac{f(y_0)}{\mnorm{y_0-z_0}^2}(y_0-z_0)},\quad\text{since $z,z_0 \in K=f^{-1}(0)$},\\
&=\skpr{z}{x_0}+\skpr{\lambda(y_0-z_0)}{x_0},\quad\text{since $x_0 \perp K$},\\
&=\skpr{x}{x_0}.
\end{align*}
\end{proof}


\subsection{Conclusion}

The implication \ref{Hilbert-original}$\Rightarrow$\ref{Hilbert-coproximinal} from \cref{result:characterization-hilbert} can be found in \cite[Theorem~1]{FranchettiFu1972}. (In fact, using the reformulation from \cref{result:projection} it is folklore.) For \ref{Hilbert-coproximinal}$\Rightarrow$\ref{Hilbert-original}, we combine \cref{result:necessity-inner-product} with \cref{result:necessity-complete}.

\textbf{Acknowledgements.} T.J. would like to acknowledge support by the DFG Ul-403/2-1 and NUTRICON project.

\providecommand{\bysame}{\leavevmode\hbox to3em{\hrulefill}\thinspace}
\providecommand{\href}[2]{#2}

\end{document}